\documentclass[12pt, twoside, a4paper]{amsart}
\usepackage{amsmath,amsthm,amsopn,amssymb,a4wide,varioref,amscd}

\usepackage{color}
\usepackage{amsthm}
\theoremstyle{plain}

\usepackage{thmtools,thm-restate}
\usepackage{bm}
\usepackage{tikz-cd}
\usepackage{hyperref}
\newcommand{\comment}[1]{}

\theoremstyle{theorem}
    \newtheorem{theorem}{Theorem}
    \newtheorem{lemma}[theorem]{Lemma}

\newtheorem*{theorem*}{Theorem}

\newcommand\tr{{\mbox{tr}}}

\newcommand\mnote[1]{} %off
\newcommand\be{\begin{equation*}}

\newcommand\ee{\end{equation*}}

\newcommand\ben{\begin{equation}}
\newcommand\een{\end{equation}}
\newcommand\bes{\begin{eqnarray*}}
\newcommand\ees{\end{eqnarray*}}

\newcommand\bex{\begin{exercise}}
\newcommand\eex{\end{exercise}}
\newcommand\beg{\begin{example}}
\newcommand\eeg{\end{example}}
\newcommand\benu{\begin{enumerate}}
\newcommand\eenu{\end{enumerate}}
\newcommand\beit{\begin{itemize}}
\newcommand\eeit{\end{itemize}}
\newcommand\berk{\begin{remark}}
\newcommand\eerk{\end{remark}}
\newcommand\bdefn{\begin{defintion}}
\newcommand\edefn{\end{definition}}
\newcommand\bthm{\begin{theorem}}
\newcommand\ethm{\end{theorem}}
\newcommand\bprf{\begin{proof}}
\newcommand\eprf{\end{proof}}
\newcommand\blem{\begin{lemma}}
\newcommand\elem{\end{lemma}}

\newcommand{\sm}{{\raise0.3ex\hbox{$\scriptstyle \setminus$}}}

\def\CHI{\mathchoice%
{\raise2pt\hbox{$\chi$}}%
{\raise2pt\hbox{$\chi$}}%
{\raise1.3pt\hbox{$\scriptstyle\chi$}}%
{\raise0.8pt\hbox{$\scriptscriptstyle\chi$}}}
\def\smalloplus{\raise1pt\hbox{$\,\scriptstyle \oplus\;$}}

\newtheorem{thm}{Theorem}[section]
\newtheorem{defn}[thm]{Definition}

\newtheorem{lem}[thm]{Lemma}

\newtheorem{rem}[thm]{Remark}

\numberwithin{equation}{section}

\def\textmatrix#1&#2\\#3&#4\\{\bigl({#1 \atop #3}\ {#2 \atop #4}\bigr)}
\def\dispmatrix#1&#2\\#3&#4\\{\left({#1 \atop #3}\ {#2 \atop #4}\right)}

\begin{document}

\title[Interpolating sequences and Toeplitz corona theorem]{Interpolating sequences and the Toeplitz corona theorem on the symmetrized bidisk }

\author[Bhattacharyya]{Tirthankar Bhattacharyya}
\address[Bhattacharyya]{Department of Mathematics, Indian Institute of Science, Bangalore 560 012, India.}
\email{tirtha@iisc.ac.in}

\author[Sau]{Haripada Sau}
\address[Sau]{Department of Mathematics, Virginia Tech, Blacksburg, VA 24061-0123, USA.}
\email{sau@vt.edu, haripadasau215@gmail.com}
\thanks{This research is supported by University Grants Commission, India via CAS. Work of the second author was largely done at the Indian Institute of Technology Bombay.}
\date{\today}
\maketitle

\begin{abstract}
This paper is a continuation of work done in \cite{BS}. It contains two new theorems about bounded holomorphic functions on the symmetrized bidisk -- a characterization of interpolating sequences and a Toeplitz corona theorem.
\end{abstract}
%\tableofcontents

\section{Statement of main results}

\subsection{Introduction}
This paper extends two important results in operator theory and complex function theory on the
unit disk to the symmetrized bidisk
$$\mathbb{G}=\{(z_1+z_2,z_1z_2): |z_1|<1, |z_2|<1\},$$
namely, a characterization of interpolating sequences (Theorem 1) and the Toeplitz-corona theorem (Theorem 2).  These two seemingly uncorrelated results are unified by the fact that both are applications of the statement and the method of proof of the Realization Theorem for operator-valued bounded holomorphic functions on $\mathbb G$ of norm no larger than $1$. The symmetrized bidisk is non-convex, but polynomially convex. It is interesting to both complex analysts and operator theorists -- for dilation and related results (\cite{ay-Edinburgh}, \cite{BPSR}), for a rich function theory (\cite{AYRealization}, \cite{BS}) and for its complex geometry (\cite{ay-geometry}, \cite{KZ}).

Let $\mathcal L$ be a Hilbert space. All Hilbert spaces in this note are separable and are over the complex field. Let $\mathcal{B}(\mathcal L)$ denote the algebra of all bounded operators acting on $\mathcal L$. A function $k:\mathbb G \times \mathbb G\to \mathcal B(\mathcal L)$  is called  {\it{positive semi-definite}} if for all $n \ge 1$, all $\lambda_1, \lambda_2, \ldots , \lambda_n$ in $\mathbb G$ and all $h_1, h_2, \ldots , h_n$ in $\mathcal L$, it is true that
\begin{eqnarray}\label{kercond}
\sum_{i,j=1}^n \langle k(\lambda_i,\lambda_j)h_i , h_j \rangle \geq0.
\end{eqnarray} If moreover, $k$ is holomorphic in the first variable, anti-holomorphic in the second variable and $k(\lambda,\lambda)\neq 0$ for every $\lambda\in \mathbb G$, then it is called a $kernel$.

A {\it{weak kernel}} $k$ is a function $k:\mathbb G \times \mathbb G\to \mathcal B(\mathcal L)$ that is holomorphic in the first variable and anti-holomorphic in the second such that (\ref{kercond}) holds with no requirement of being non-zero on the diagonal.

In what follows, a kernel or a weak kernel will be assumed to be scalar-valued, i.e., when $\mathcal L=\mathbb C$, unless otherwise mentioned.

It is elementary that for every $\mathcal B(\mathcal L)$-valued positive semi-definite function $k$, there is a Hilbert space $H_k$ consisting of $\mathcal L$-valued functions on $\mathbb G$ such that the set of functions
\begin{equation} \label{kernelspace}
\{\sum_{i=1}^n k(\cdot, \lambda_i)h_i: n \in \mathbb N, \; h_i\in \mathcal L \text{ and }\lambda_i\in \mathbb G\}
 \end{equation}
 is dense in $H_k$ and $\langle f, k(\cdot,\lambda)h\rangle_{H_k}=\langle f(\lambda), h\rangle_{\mathcal L}$ for any $f \in H_k$, $h \in \mathcal L$ and $\lambda \in \mathbb G$. If moreover, $k$ is a kernel (or a weak kernel), then the functions are holomorphic. The operator theory comes into play because of the following.

Let $(T_1, T_2)$ be a pair of commuting bounded operators acting on a Hilbert space and $\sigma(T_1,T_2)$ be the Taylor joint spectrum of $(T_1,T_2)$. A polynomially  convex compact set $X \subseteq \mathbb C^2$ is called a {\em{spectral set}} for $(T_1, T_2)$, if $\sigma (T_1, T_2) \subseteq X$ and
$$ \| \xi(T_1, T_2) \| \le \sup_X | \xi |$$
for any polynomial $\xi$ in two variables.

A typical point of the symmetrized bidisk will be denoted by $(s,p)$. The terminology {\it$\Gamma$--contraction} in the following definition is by now classical and was introduced by Agler and Young in \cite{ay-jot}.

\begin{defn} \label{Gamma--contraction}
A pair of commuting bounded operators $(S, P)$ on a Hilbert
space $\mathcal H$ having the closed symmetrized bidisk $\Gamma$
as a spectral set is called a $\Gamma$--contraction. Thus $(S,P)$
is a $\Gamma$--contraction if and only if $\| \xi (S,P) \| \le
\sup_{\mathbb G} |\xi|$ for all polynomials $\xi$ in two variables.

 A $\mathcal B(\mathcal L)$-valued kernel $k$ on $\mathbb G$ is called admissible if the pair $(M_s,M_p)$ of multiplication by the co-ordinate functions is a $\Gamma$--contraction on $H_k$.

 Similarly, a kernel $k$ on $\mathbb D^2$ ($\mathbb D$ being the unit disk in the complex plane) is called admissible if the multiplication operators $M_{z_1}$ and $M_{z_2}$ by the coordinate functions are contractions on $H_k$. \end{defn}

 Note that the definition of admissibility is attuned to the domain. For the bidisk, we demand that the operator pair of multiplication by the coordinate functions has $\overline{\mathbb D^2}$ as a spectral set whereas for the symmetrized bidisk, the demand is that the operator pair of multiplication by the coordinate functions has $\Gamma$ as a spectral set.

\subsection{Interpolating Sequences} \label{IS}

A sequence $\{(s_j,p_j): j \geq 1\}$ of points in $\mathbb{G}$ is called an {\em interpolating sequence} for $H^\infty(\mathbb{G})$, the algebra of bounded analytic functions on $\mathbb{G}$, if for every bounded sequence $w=\{w_j: j \geq 1\}$ of complex numbers, there exists a function $f$ in $H^\infty(\mathbb{G})$ such that $f(s_j,p_j)=w_j$ for each $j\geq 1$. Interpolating sequences for the algebra $H^\infty(\mathbb D)$ of bounded analytic functions on $\mathbb D$ were characterized by Carleson \cite{Carleson-IntSeq}. One of his characterizations of interpolating sequences is that a sequence $\{z_j\}$ in $\mathbb D$ is interpolating if and only if there exists $\delta > 0$ such that
$$
\prod_{j\neq k}\left|\frac{z_j-z_k}{1-\overline{z_k}z_j}\right| \geq \delta, \text{ for all } k.
$$
In \S 3, we shall see that there exist uncountably many Carleson-type sufficient conditions for a sequence in $\mathbb G$ to be interpolating (Lemma \ref{suff}).

A sequence $\{(s_j,p_j):j\geq 1\}$ of points in $\mathbb G$ is called {\em strongly separated} if there exists a constant $M$ such that, for each $i$ there is an $f_i$ in $H^\infty (\mathbb G)$ of norm at most $M$ that satisfies $f_i(s_i,p_i ) = 1$ and $f_i(s_j,p_j ) = 0$ for all $j$  other than $i$. And the sequence is called {\em weakly separated} if whenever $i \neq j$, there exists a function $f_{ij}$ in $H^\infty (\mathbb G)$ of norm at most $M$ that satisfies $f_{ij}(s_i,p_i) = 1$ and $f_{ij}(s_j,p_j) = 0$.

Note that an interpolating sequence is strongly separated and  a strongly separated sequence is weakly separated.

For a given sequence $\{(s_j,p_j): j \geq 1\}$ in $\mathbb{G}$ and a kernel $k$ on $\mathbb{G}$, the {\em normalized Grammian} of $k$ is the following infinite matrix
$$
G_k=\left(  \frac{k((s_i,p_i),(s_j,p_j))}{\sqrt{k((s_i,p_i),(s_i,p_i))}\sqrt{k((s_j,p_j),(s_j,p_j))}}  \right)_{i,j=1}^\infty.
$$
The following theorem characterizes the interpolating sequences on the symmetrized bidisk, which will be proved in \S 3. Interpolating sequences on the bidisk were characterized in \cite{Ag-Mc IntSeq}.
\begin{theorem}[Characterization of Interpolating Sequences]
Let $\{(s_j,p_j): j \geq 1\}$ be a sequence in $\mathbb{G}$. Then the following are equivalent:
\begin{enumerate}
\item[(i)] The sequence $\{(s_j,p_j): j \geq 1\}$ is an interpolating sequence for $H^\infty(\mathbb{G})$;
\item[(ii)]For all admissible kernels $k$, the normalized Grammians are uniformly bounded below, i.e., there exists an $N>0$ such that
\begin{eqnarray}\label{grammbelow}
G_k \geq \frac{1}{N}I
\end{eqnarray}for every admissible kernel $k$;
\item[(iii)] The sequence $\{(s_j,p_j): j \geq 1\}$ is strongly separated and for all admissible kernels $k$, the normalized Grammians are uniformly bounded above, i.e., there exists an $M>0$ such that
\begin{eqnarray}\label{grammabove}
G_k \leq MI
\end{eqnarray} for every admissible kernel $k$.
\item[(iv)] Items (ii) and (iii) both hold.

\end{enumerate} \label{Interpolate}
\end{theorem}

Although (iv) is redundant in view of (ii) and (iii), we have listed it because in the course of the proof, we shall first show that (i) is equivalent to (iv), and shall then show that (ii) is equivalent to (iii). It is clear that (i) and (iv) together are equivalent to (ii) and (iii) together.

\subsection{Toeplitz Corona Theorem}

The Corona Theorem for $H^\infty(\mathbb D)$ is a statement about its maximal ideal space. Obviously, $\mathbb D$ is contained in the maximal ideal space $M_{H^\infty(\mathbb D)}$ of the Banach algebra $H^\infty(\mathbb D)$ by means of identification of a $w\in \mathbb D$ with the multiplicative linear functional of evaluation, $f\to f(w)$ for all $f\in H^\infty(\mathbb D)$. It is usually difficult to find the maximal ideal space of a Banach algebra. Kakutani asked whether the {\em{corona}} $M_{H^\infty(\mathbb D)}\smallsetminus \overline{\mathbb D}$ (in the weak-star topology) is empty or in other words, whether $\mathbb D$ is dense in $M_{H^\infty(\mathbb D)}$ in the natural weak-star topology. Elementary functional analysis shows that Kakutani's question is equivalent to the following:

Given $\varphi_1$, $\varphi_2$, \dots, $\varphi_N$ in $H^\infty(\mathbb D)$ satisfying
\begin{eqnarray}\label{coronacondition}
|\varphi_1(z)|^2+|\varphi_2(z)|^2+\cdots+|\varphi_N(z)|^2 \geq \delta^2 >0 \text{ for all $z \in \mathbb{D}$},
\end{eqnarray}for some $\delta >0$, is it true that there are functions $\psi_1,\psi_2,\dots,\psi_N$ in $H^\infty (\mathbb{D})$ such that \begin{eqnarray}\label{coronaconclusion}
\psi_1\varphi_1+\psi_2\varphi_2+\cdots+\psi_N\varphi_N = 1?
\end{eqnarray}
It is easy to see that the converse implication is true, so that (\ref{coronacondition}) is a necessary condition for (\ref{coronaconclusion}). The sufficiency was proved, and hence Kakutani's question was answered affirmatively by Carleson \cite{Carleson}. This triggered a rather long list of research work on issues related to the corona theorem. First, H$\ddot{\text{o}}$rmander \cite{Hormander} introduced a different approach based on an appropriate inhomogeneous $\bar{\partial}$-equation, see \cite{corona survey} and references therein for a beautiful discussion and various results in this direction. Then Wolff produced a simpler proof than Carleson's, see \cite{corona survey} for Wolff's solution. Coburn and Schechter in \cite{Coburn-Schechter} and Arveson in \cite{Arveson} came up with an operator inequality to replace (\ref{coronacondition}):
\begin{eqnarray}\label{T-coronacondition}
M_{\varphi_1}M_{\varphi_1}^*+M_{\varphi_2}M_{\varphi_2}^*+\cdots+M_{\varphi_N}M_{\varphi_N}^* \geq \delta^2 >0,
\end{eqnarray}where the notation $M_\varphi$, for a $\varphi \in H^\infty(\mathbb D)$ stands for the operator of multiplication by $\varphi$ on $H^2(\mathbb D)$ and is also called the Toeplitz operator with symbol $\varphi$. Coburn and Schechter were interested in interpolation (in other words when an ideal in a Banach algebra contains the identity) whereas Arveson's motivation was to search for an operator theoretic proof of the corona theorem. Using the Szeg\H{o} kernel, it is elementary to see that (\ref{T-coronacondition}) implies (\ref{coronacondition}). Both papers mentioned above proved that (\ref{T-coronacondition}) implies (\ref{coronaconclusion}). Arveson achieved a bound: if $\|\varphi_i\|_\infty\leq 1$ for each $1\leq i\leq N$, then $\psi_1, \psi_2,\dots,\psi_N$ can be so chosen that
$$
\|\psi_i\|_\infty\leq 4N\epsilon^{-3}.
$$
Using the Corona Theorem, one can show that (\ref{coronaconclusion}) implies (\ref{T-coronacondition}). Thus, in the disk, all three statements are equivalent. In a general domain, equivalence of (\ref{coronaconclusion}) and (\ref{T-coronacondition}) is called the Toeplitz Corona Theorem. Agler and McCarthy proved the Toeplitz corona theorem for the bidisk in \cite{ag-Mc NP}. Our second contribution in this note is the Toeplitz corona theorem for the symmetrized bidisk.
%We shall see below that the Toeplitz corona theorem on these two domains have in general no connection apart from a couple of necessary conditions. We shall also analyse when these necessary conditions are sufficient.

Before we state the theorem, we need to make a few comments about $H^\infty(\mathbb G)$. Recall that a scalar-valued kernel $k$ on $\mathbb G \times \mathbb G$ is admissible if the pair $(M_s, M_p)$ of multiplication operators on $H_k$ forms a $\Gamma$--contraction. There is a characterization of $H^\infty(\mathbb G)$ through admissible kernels: a function $\varphi$ in $H^\infty(\mathbb G)$ has norm no larger than $1$ if and only if $M_\varphi$ is a contraction on $H_k$ for every admissible kernel $k$ on $\mathbb G \times \mathbb G$. In other words, a function $\varphi$ is in $H^\infty(\mathbb G)$ if and only if the operator $M_\varphi$ is a bounded operator on $H_k$ for all admissible kernels $k$. We refer the reader to Lemma 3.1 of \cite{BS} for the proof.

Buoyed by this fact, we may ask whether the admissible kernels can be replaced by measures on the distinguished boundary $b\Gamma$ of the symmetrized bidisk. This is the Shilov boundary with respect to the algebra of functions continuous on $\Gamma$ and holomorphic on $\mathbb G$. It turns out that
$$b\Gamma = \{ (z_1 + z_2, z_1z_2): |z_1| = |z_2| = 1\},$$
see Theorem 1.3 in \cite{ay-jot}.
For a regular Borel measure $\mu$ on $b\Gamma$ (respectively on $\mathbb{T}^2$), let $H^2(b\Gamma,\mu)$ (respectively $H^2(\mathbb{T}^2,\mu)$) denote the closure of all polynomials in $L^2(b\Gamma,\mu)$ (respectively $L^2(\mathbb{T}^2,\mu)$). For a function $\varphi \in H^\infty(\mathbb{G})$, we consider its radial limit, which exists almost everywhere with respect to the Lebesgue measure in $b\Gamma$, and denote it by $\varphi$ itself. However, $M_\varphi$ need not be defined as a bounded operator on $H^2(b\Gamma,\mu)$. To remedy this situation, consider the following scaling of the function $\varphi$:
$$
\varphi_r(s,p):=\varphi(rs,r^2p)\text{ for all } (s,p)\in \mathbb{G} \text{ and }0\leq r<1.
$$
Now, $M_{\varphi_r}$ is a bounded operator on $H^2(\mathbb{G},\mu)$. For two Hilbert spaces $\mathcal L_1$ and $\mathcal L_2$, let $H^\infty(\mathbb G, \mathcal B (\mathcal L_1, \mathcal L_2))$ denote the Banach space of $\mathcal B (\mathcal L_1, \mathcal L_2)$-valued bounded analytic functions on $\mathbb G$. Given an admissible kernel $k$ on $\mathbb G$ and $\Phi \in H^\infty(\mathbb G, \mathcal B (\mathcal L_1, \mathcal L_2))$, it  is natural to consider the multiplication operator $M_\Phi$ from the Hilbert space $H_k(\mathbb G) \otimes \mathcal L_1$ (identified as a Hilbert space of $\mathcal L_1$-valued functions) into $H_k(\mathbb G) \otimes \mathcal L_2$. Similarly, $M_\Phi^\mu$ will denote the multiplication operator from $H^2(b\Gamma,\mu) \otimes \mathcal L_1$ into $H^2(b\Gamma,\mu) \otimes \mathcal L_2$.

Finally, given a domain $\Omega$ and two functions $k_1,k_2:\Omega\times \Omega \to B(\mathcal L)$, we follow Agler and McCarthy, \cite{ag-Mc}, to use the notation $k_1 \oslash k_2$ for the $B(\mathcal L \otimes \mathcal L)$-valued function on $\Omega\times \Omega $ defined by
$$
k_1 \oslash k_2(z,w)=k_1(z,w) \otimes k_2(z,w)
$$ for all $z$, $w$ in $\Omega$.

\vspace*{5mm}
\begin{theorem}[Toeplitz Corona Theorem on the Symmetrized Bidisk]\label{TC-G}
Let $\Phi$ be a function in $H^\infty(\mathbb G, \mathcal B (\mathcal L_1, \mathcal L_2))$  and let $\delta $ be a positive number. The following statements are equivalent.
\begin{enumerate}
\item[(1)]There is a $\Psi \in H^\infty(\mathbb G, \mathcal B(\mathcal L_2, \mathcal L_1))$ of norm no larger than $1/\delta$ such that
    \begin{align}\label{LeftInverse}
    \Psi (s,p) \Phi (s,p) = I_{\mathcal L_2}
    \end{align} for all $(s,p) \in \mathbb G$.
\item[(2)]For every regular Borel measure $\mu$ on $b\Gamma$, the operator
$$ M_{\Phi_r}^\mu (M_{\Phi_r}^\mu)^* - \delta I_{H^2(b\Gamma,\mu) \otimes \mathcal L_2}$$
is positive for every $0< r < 1$.
\item[(3)]For any $\mathcal B(\mathcal L_2)$-valued admissible kernel $k$ on $\mathbb G$, the function
$$(\Phi (s,p) \Phi(t,q)^* - \delta I_{\mathcal L_2})\oslash k((s,p),(t,q))
    $$
    is positive semi-definite.
\end{enumerate}
\end{theorem}

This theorem will be proved in \S 5.

At this stage, one may wonder whether the Toeplitz corona theorem on the symmetrized bidisk follows from that on the bidisk. We explain below why it does not.

Let us start with a $\Phi$ satisfying the condition (1) in the theorem above. In bidisk coordinates, the equation (\ref{LeftInverse}) can be written as
$$
\Phi(\gamma(z_1,z_2))\Psi(\gamma(z_1,z_2))=I_{\mathcal L_2},
$$where $\gamma: \mathbb D^2 \rightarrow \mathbb G$ is the symmetrization map $\gamma(z_1,z_2)=(z_1+z_2,z_1z_2)$. A straightforward application of Theorem 11.65 in \cite{ag-Mc} implies the following two facts:

\begin{enumerate}
\item[(2$'$)] For every $\mathcal{B}(\mathcal{L}_2)$-valued admissible kernel $k$ on $\mathbb D^2$, the function
    $$ \big((\Phi \circ \gamma(z_1,z_2))({\Phi \circ \gamma(w_1,w_2)})^* - \delta  I_{\mathcal L_2}\big) \oslash k((z_1,z_2),(w_1,w_2))$$ on $\mathbb{D}^2\times\mathbb{D}^2$ is positive semi-definite; or equivalently,
\item[(3$'$)] For every regular Borel measure $\mu$ on $\mathbb T^2$ and $0< r < 1$, the operator
    $$ M^\mu_{\Phi \circ \gamma_r} (M^\mu_{\Phi \circ \gamma_r})^* - \delta  I_{H^2(b\Gamma,\mu) \otimes \mathcal L_2}$$ is positive, where $\gamma_r:\mathbb D^2\to \mathbb G$ is the map
    $$\gamma_r(z_1,z_2)=(rz_1+rz_2,r^2z_1z_2).$$
\end{enumerate}

Moreover, (2$'$) and (3$'$) are equivalent. The challenge is to get a left inverse $\Psi$ of $\Phi$ if (2$'$) and (3$'$) hold. What one can get is a function $\Psi$ in $H^\infty(\mathbb D^2, \mathcal B(\mathcal L_2, \mathcal L_1))$ such that
\begin{align}\label{Con-Nesconds}
\Phi(\gamma(z_1,z_2))\Psi(z_1,z_2)=I_{\mathcal L_2}.
\end{align}
This is due to Theorem 11.65 in \cite{ag-Mc} again. However, $\Psi$ need not be symmetric, i.e., $\Psi(z_1,z_2)$ need not be the same operator as $\Psi(z_2,z_1)$ and hence, in general, it does not give rise to a function on the symmetrized bidisk. One case when the necessary conditions (2$'$) and (3$'$) will be sufficient as well is when  $\Phi(s,p)$ is a one-one operator  for every $(s,p)$ in $\mathbb G$. We leave it to the reader to check that $\Psi$ turns out to be symmetric in this case.

This shows the need of proving the Toeplitz corona theorem for the symmetrized bidisk separately although the theorem has been well-established in the bidisk.
In fact, it has been observed many times that results in the bidisk do not imply results in the symmetrized bidisk, naive attempts to deduce results for the symmetrized bidisk from the corresponding results for the bidisk run into difficulty. The proof of the existence of rational dilation in the symmetrized bidisk (see \cite{ay-Edinburgh}) is an example: it needed a substantial amount of effort and tools while the rational dilation theorem was known to succeed in the bidisk due to And\^o \cite{ando}. This is the case for the two theorems stated above too. There are at least two reasons why this happens: the admissible kernels on the symmetrized bidisk have no relation to the admissible kernels on the bidisk and there are uncountably infinitely many parametrized co-ordinate functions (to be defined in the next section) on the symmetrized bidisk as opposed to only two on the bidisk.

We conclude this section by noting the results of Amar in \cite{Amar} where he proved a Toeplitz Corona theorem for a bounded convex domain in $\mathbb C^n$ in terms of measures on the boundary. Results for the symmetrized bidisk do not follow from Amar's results because the  symmetrized bidisk is not convex.

\section{Background on The Realization Theorem}

One of the most important results in the area of holomorphic functions and in the theory of Hilbert space operators is the realization formula. A function $f$ is in $H^\infty(\mathbb D)$ and satisfies $\| f \|_\infty \le 1$ if and only if there is a Hilbert space $\mathcal H$ and a unitary operator
$$U = \dispmatrix A & B \\ C & D \\ : \mathbb C \oplus \mathcal H \rightarrow \mathbb C \oplus \mathcal H$$ such that
$$ f(z) = A + zB (I - zD)^{-1} C.$$
Agler generalized this elegantly to the bidisk in \cite{ag}. He showed that a function $f$ is in $H^\infty(\mathbb D^2)$ and satisfies $\| f \| \le 1$ if and only if there is a graded Hilbert space $\mathcal H = \mathcal H_1 \oplus \mathcal H_2$ and a unitary operator
$$U = \dispmatrix A & B \\ C & D \\ : \mathbb C \oplus \mathcal H \rightarrow \mathbb C \oplus \mathcal H$$ such that writing $P_i$ for the projection from $\mathcal H$ onto $\mathcal H_i$ for $i=1,2$, we have
$$ f(z) = A + B (z_1 P_1 + z_2 P_2) (I - D (z_1 P_1 + z_2 P_2) )^{-1} C.$$
The importance of realization formulae lie in their applications to several interesting areas of research including Pick interpolation, Beurling type submodules and the corona problem, see \cite{ag-Mc NP}, \cite{BL}, \cite{BK} and \cite{DM} and the book \cite{ag-Mc}.

 The Realization Theorem for the symmetrized bidisk was proved in \cite{AYRealization} and \cite{BS} for scalar-valued functions. Here a version of it for operator-valued functions is needed which we shall state and {\em{not}} prove because all the crucial concepts of the proof are present in the scalar case and hence the same proof with necessary modifications continues to hold in the case when the functions are operator-valued.

We shall need one more level of generalization of the concept of kernels than what has already been explained. For two $C^*$-algebras $\mathcal A$ and $\mathcal C$, a function $\Delta: \mathbb G\times \mathbb G\to \mathcal B(\mathcal A,\mathcal C)$ is called a completely positive function  if
$$
\sum_{i,j=1}^Nc_i^*\Delta\big( \lambda_i, \lambda_j \big)(a_i^*a_j)c_j
$$is a  non-negative element of $\mathcal C$ for any positive integer $N$, any $n$ points $\lambda_1, \lambda_2,\dots, \lambda_N$ of $\mathbb G$, any $N$ elements $a_1,a_2,\dots,a_N$ from $\mathcal A$ and any $N$ elements $c_1,c_2,\dots,c_N$ from $\mathcal C$. We give an example of such a completely positive function. Let $\delta:\overline{\mathbb D}\times \mathbb G\times \mathbb G\to \mathcal B(\mathcal L)$ be a function such that for each $\alpha \in \overline{\mathbb D}$, the function $\delta(\alpha, \cdot, \cdot)$ is a positive semi-definite function on $\mathbb G$ and for every fixed $(s,p)$ and $(t,q)$ in $\mathbb G$, the function $\delta(\cdot, (s,p), (t,q))$ is a Borel measurable function on $\overline{\mathbb D}$. Given a positive regular Borel measure  $\mu$ on $\overline{\mathbb D}$, the function $\Delta^\delta_\mu:\mathbb G \times \mathbb G \to \mathcal B(C(\overline{\mathbb D}), \mathcal B(\mathcal L))$ defined by
\begin{equation} \label{ExampleofDelta}
\Delta^\delta_\mu\left( (s,p), (t,q) \right) (h) = \int_{\overline{\mathbb D}} h(\cdot) \delta(\cdot , (s,p), (t,q) ) d\mu
\end{equation}
is a completely positive function on $\mathbb G$. More details on these functions are found in \cite{BBLS}.

When we use the word kernel or the phrase weak kernel, holomorphy in the first component and anti-holomorphy in the second component are built in whereas when we use the word function, as in a completely positive function, no such holomorphy is assumed.

For $\alpha \in \overline{\mathbb D}$ and $(s,p) \in \mathbb G$, let
\begin{eqnarray}\label{testfunction}
 \varphi(\alpha, s , p) = \frac{2 \alpha p - s}{2 - \alpha s}.
\end{eqnarray}
Since $|s| < 2$ for all $(s,p) \in \mathbb G$ and $\alpha \in \overline{\mathbb D}$, this function is well-defined on $\overline{\mathbb D} \times \mathbb G$. Agler and Young proved (Theorem 2.1, \cite{ay-geometry}) that
\begin{equation} \label{co-ordinates} (s,p) \in \mathbb G \mbox{ if and only if } \varphi(\alpha, s, p) \in \overline{\mathbb D} \end{equation}
for all $\alpha$ in the closed unit disk. For this reason, we call the family $\{\varphi(\alpha, \cdot): \alpha \in \overline{\mathbb D}\}$ the {\em parametrized coordinate functions} for the symmetrized bidisk. We note that for every $\alpha \in \overline{\mathbb D}$, the function $\varphi(\alpha , \cdot)$ is in the norm unit ball of $H^\infty(\mathbb G)$, and for every $(s,p) \in \mathbb G$, the function $\varphi( \cdot , s, p)$ is in $C(\overline{\mathbb D})$. The following lemma gives an equivalent formulation of admissiblility of a kernel $k$ on $\mathbb{G}$ in terms of coordinate functions. See Lemma 3.2 of \cite{BS} for a proof of this.
\begin{lem} \label{admi}
A $\mathcal{B}(\mathcal{L})$-valued kernel $k$ on $\mathbb G$ is admissible if and only if the following $\mathcal{B}(\mathcal{L})$-valued function
$$ (1 - \varphi(\alpha, s, p) \overline{\varphi(\alpha, t, q)}) k\big( (s,p) , (t,q) \big)$$
 on $\mathbb{G}\times\mathbb{G}$ is positive semi-definite for every $\alpha \in \overline{\mathbb D}$. \end{lem}

\noindent \textbf{Realization Theorem for Operator-Valued Functions.} {\em
Let $\mathcal L_1$ and $\mathcal L_2$ be two Hilbert spaces, $Y$ be any subset of $\mathbb G$ and $f:Y \to \mathcal{B}(\mathcal L_1, \mathcal L_2)$ be any function. Then the following statements are equivalent.

\begin{description}
        \item[(H)] There exists a function $F$ in $H^\infty(\mathbb G,\mathcal{B}(\mathcal L_1, \mathcal L_2))$ with $\| F \|_\infty \le 1$ and $F|_Y=f$;
        \item[(M)] The function
        $$((s,p),(t,q))\mapsto (I_{\mathcal L_2} - f(s,p)f(t,q)^*) \oslash  k( (s,p) , (t,q)) $$ is a weak kernel for every $\mathcal B(\mathcal L_2)$-valued admissible kernel $k$ on $Y$;
        \item[(D)] There is a completely positive function $\Delta: Y \times  Y \to \mathcal B\big(C(\overline{\mathbb D}), \mathcal B(\mathcal L_2)\big)$ such that for every $(s,p)$ and $(t,q)$ in $Y$,
        $$ I_{\mathcal L_2} - f(s,p)f(t,q)^* = \Delta ( (s,p), (t,q)) \big( 1 - \varphi(\cdot, s, p) \overline{\varphi(\cdot , t, q)} \big) ;$$
        \item[(R)] There is a Hilbert space $\mathcal H$, a unital $*$-representation $\pi : C(\overline{\mathbb D}) \rightarrow B(\mathcal H)$ and a unitary $V : \mathcal L_1 \oplus \mathcal H \rightarrow \mathcal L_2 \oplus \mathcal H$ such that writing $V$ as
            $$ V = \left(
                     \begin{array}{cc}
                       A & B \\
                       C & D \\
                     \end{array}
                   \right)$$
                   we have $f(s,p) = A +  B \pi(\varphi(\cdot, s, p)) \big( I_\mathcal H -  D \pi(\varphi(\cdot, s, p)) \big)^{-1} C$, for every $(s,p)\in Y$.
      \end{description}}

\section{Interpolating sequences -- Proof of Theorem \ref{Interpolate}}

 The celebrated Pick interpolation theorem, now studied for a hundred years, characterizes those data $\lambda_1$, $\lambda_2$,$\dots$, $\lambda_N$ in $\mathbb D$ and $w_1,w_2,\dots,w_N$ in $\overline{\mathbb D}$ for which there is a function $f\in H^\infty(\mathbb D)$ interpolating the data: there is an $f\in H^\infty(\mathbb D)$ with $\|f\|_\infty\leq 1$ and $f(\lambda_i)=w_i$, $i=1,2,\dots,N$ if and only if
$$
\left(\left( \frac{1-w_i\overline{w_j}}{1-\lambda_i\overline{\lambda_j}} \right)\right)_{i,j=1}^N
$$is a positive semi-definite matrix. For detailed discussions on Pick interpolation in various contexts, see \cite{BBFt} and \cite{Bt}. In \cite{BS}, we proved the Interpolation Theorem for the symmetrized bidisk. Its version for operator-valued functions is as follows. We again omit the proof because it is similar to the proof of the scalar version in \cite{BS}.

 \vspace*{5mm}

\noindent \textbf{Interpolation Theorem for Operator-Valued Functions.} {\em
Let $\mathcal{L}_1$ and $\mathcal{L}_2$ be Hilbert spaces and $W_1,W_2,\dots,W_N \in \mathcal{B}(\mathcal{L}_1,\mathcal{L}_2)$. Let $(s_1,p_1),(s_2,p_2),\dots,(s_N,p_N)$ be $N$ distinct points in $\mathbb{G}$. Then there exists a function $f$ in the closed unit ball of $H^\infty\big(\mathbb{G},\mathcal{B}(\mathcal{L}_1,\mathcal{L}_2)\big)$ interpolating each $(s_i,p_i)$ to $W_i$ if and only if
\begin{eqnarray}\label{pickmatrix}
\big((I_{\mathcal{L}_2}-W_iW_j^*) \otimes k\big((s_i,p_i),(s_j,p_j)\big)\big)_{i,j=1}^N
\end{eqnarray}is a positive operator on $\oplus_{i=1}^N \mathcal L_2$, for every $\mathcal{B}(\mathcal{L}_2)$-valued admissible kernel $k$ on $\mathbb{G}$.}

The Interpolation Theorem mentioned above was stated in Subsection 1.2, page 508 of \cite{BS} for scalar-valued functions. The following lemma is a straightforward consequence of that and hence we leave the proof to the reader.
It deals with an infinite data set.

\begin{lem}\label{infinitePick}
Let $\{\lambda_j=(s_j,p_j): j \geq 1\}$ be a  sequence of points in $\mathbb{G}$ and $w=\{w_j: j \geq 1\}$ be a bounded sequence of complex numbers. Then there exists a function $f_w$ in $H^\infty(\mathbb{G})$ with $\|f_w\|_\infty \leq C_w$ such that $f_w(\lambda_j)=w_j$ for each $j\geq 1$ if and only if
$$ (i,j) \longrightarrow  (C_w^2 - w_i \overline{w}_j) k(\lambda_i , \lambda_j) $$
is a positive semi-definite function on $\mathbb N \times \mathbb N$ for every admissible kernel $k$ on $\mathbb{G}$.
\end{lem}

For an interpolating sequence $\Upsilon = \{(s_j,p_j): j \geq 1\}$ in $\mathbb{G}$, the following constant is called the {\em{constant of interpolation}}:
$$
\sup_{\|(w_j)\|_\infty \leq 1}\inf\{\|f\|_\infty: f \in H^\infty(\mathbb{G}), f(s_j,p_j)=w_j, j=1,2,3, \dots\}.
$$
This constant depends on $\Upsilon$ and we show below that it is finite for any interpolating sequence $\Upsilon$. To that end, define a linear operator $L_\Upsilon:H^\infty(\mathbb{G}) \to l^\infty$ by
$$ f \mapsto (f(s_1,p_1),f(s_2,p_2),f(s_3,p_3),\dots). $$
Clearly, $L_\Upsilon$ is a contraction.
Recall that the definition of an interpolating sequence (given in Subsection \ref{IS}), precisely means that $L_\Upsilon$ is onto. Let $\mathcal{N}$ be the null space of $L_\Upsilon$. Then the natural map $\tilde{L}_\Upsilon:H^\infty(\mathbb{G})/\mathcal{N} \to l^\infty$ is an isomorphism. Let $R_\Upsilon$ be the inverse of $\tilde{L}_\Upsilon$ and $w=\{w_j: j \geq 1\}$ be a sequence in $l^\infty$. Then $R_\Upsilon(w)=f_w + \mathcal{N}$, where $f_w \in H^\infty(\mathbb{G})$ is such that $f_w(s_j,p_j)=w_j$ for each $j\geq 1$. We claim that $\|R_\Upsilon\|$ is the constant of interpolation for $\Upsilon$. Indeed,
\begin{eqnarray*}
\|R_\Upsilon\| = \sup_{\|(w_j)\|_\infty \leq 1}\|f_w+\mathcal{N}\| &=& \sup_{\|(w_j)\|_\infty \leq 1} \inf\{\|f_w +g\|_\infty: g \in \mathcal{N}\}
\\
&=&\sup_{\|(w_j)\|_\infty \leq 1} \inf\{\|f\|_\infty: f(s_j,p_j)=w_j, j=1,2,3, \dots\}.
\end{eqnarray*}

The next lemma is a decomposition of a completely positive function. The proof is along the lines of Proposition 3.3 in \cite{DM} and hence we omit it.

 \begin{lem}\label{crucial-lemma}
Let $Y$ be a subset of $\mathbb G$. If $\Delta:Y \times Y \to \mathcal B\big(C(\overline{\mathbb D}), \mathcal B(\mathcal L)\big)$ is completely positive function, then there is a Hilbert space $\mathcal H$ and a function $L:Y \to B\big(C(\overline{\mathbb D}), \mathcal B(\mathcal H,\mathcal L)\big)$ such that for every $h_1$, $h_2 \in C(\overline{\mathbb D})$ and $(s,p)$, $(t,q)\in Y$,
\begin{eqnarray}\label{kernel-decomposition}
\Delta\big((s,p), (t,q)\big)(h_1\overline{h_2})=L(s,p)[h_1] (L(t,q)[h_2])^*.
\end{eqnarray}
Moreover, there is a unital $*$-representation $\pi:C(\overline{\mathbb D})\to \mathcal B(\mathcal H)$ such that for every $h_1$, $h_2 \in C(\overline{\mathbb D})$ and $(s,p) \in Y$,
\begin{eqnarray}\label{star-represt.}
L(s,p)(h_1h_2) = \pi(h_1)L(s,p)h_2.
\end{eqnarray}
\end{lem}

 We are now ready for the proof of Theorem \ref{Interpolate}

\textbf{Proof of Theorem \ref{Interpolate}}

$(i)\Rightarrow (iv)$: Let $\{(s_j,p_j): j \geq 1\}$ be an interpolating sequence for $H^\infty(\mathbb{G})$ with constant of interpolation $M$. This means for every $w=(w_j)$ with $\|w\|_\infty \leq 1$, there exists a function $f$ in $H^\infty(\mathbb{G})$ such that $f(s_j,p_j)=w_j$ and $\|f\|_\infty\leq M$. Therefore by Lemma \ref{infinitePick} we have
\begin{eqnarray}\label{Pickequivalent}
\sum_{i,j=1}^nc_i\overline{c_j}w_i\overline{w_j}k((s_i,p_i),(s_j,p_j))\leq M^2\sum_{i,j=1}^nc_i\overline{c_j}k((s_i,p_i),(s_j,p_j)),
\end{eqnarray} for any $n \in \mathbb N$ and any complex numbers $c_1, c_2, \ldots ,c_n$. Now the proof depends on choosing $w_j$ and $c_j$ appropriately.
Choose $w_j=\exp(i \theta_j)$ and let $(c_j)$ be any sequence in $l^2$. Then we have
$$
\sum_{i,j=1}^nc_i\overline{c_j}\exp( i (\theta_i-\theta_j))k((s_i,p_i),(s_j,p_j))\leq M^2\sum_{i,j=1}^nc_i\overline{c_j}k((s_i,p_i),(s_j,p_j)),
$$which, after integrating with respect to $\theta_1,\theta_2,\dots ,\theta_n$ on $[0,2\pi]\times[0,2\pi]\times \cdots \times [0,2\pi]$ and dividing by $(2\pi)^n$ both sides, becomes
$$
\sum_{j=1}^n|c_j|^2k((s_j,p_j),(s_j,p_j))\leq M^2\sum_{i,j=1}^nc_i\overline{c_j}k((s_i,p_i),(s_j,p_j)),
$$which after replacing $c_j$ by $c_j':=c_j/\sqrt{k((s_j,p_j),(s_j,p_j))}$ becomes
$$ \sum_{j=1}^n|c_j|^2\leq M^2\sum_{i,j=1}^nc_i\overline{c_j}G_k(i,j). $$
Similarly, choosing $w_j=\exp(i \theta_j)$ and $c_j'':=\exp(-i \theta_j)c_j/\sqrt{k((s_j,p_j),(s_j,p_j))}$ in (\ref{Pickequivalent}) and integrating as above we get
$$
\sum_{i,j=1}^nc_i\overline{c_j}G_k(i,j) \leq M^2 \sum_{j=1}^n|c_j|^2.
$$ Consequently, whenever we have an interpolating sequence $\{(s_j,p_j): j \geq 1\}$ with $M$ as its constant of interpolation, we have for every admissible kernel $k$
\begin{eqnarray}\label{Grammcond}
\frac{1}{M^2}\sum_{j=1}^n|c_j|^2\leq \sum_{i,j=1}^nc_i\overline{c_j}G_k(i,j)\leq M^2 \sum_{j=1}^n|c_j|^2,
\end{eqnarray} where $G_k(i,j)$ is the $(ij)$-th entry of the Grammian matrix associated to $k$ and the interpolating sequence. Since this is true for every $n \ge 1$, we have shown that the constants $M$ and $N$ in (\ref{grammabove}) and (\ref{grammbelow}) can be chosen to be the square of the constant of interpolation.

$(iv)\Rightarrow (i)$: Conversely, suppose (\ref{grammabove}) and (\ref{grammbelow}) hold for some constants $M$ and $N$. Without loss of generality we can assume that $M$ and $N$ are the same. Therefore for every admissible kernel $k$ and $(c_j)$ in $l^2$, we have
\begin{eqnarray}\label{theconverse}
\frac{1}{M}\sum_{j}|c_j|^2\leq \sum_{i,j}c_i\overline{c_j}G_k(i,j)\leq M \sum_{j}|c_j|^2.
\end{eqnarray}
To prove that $\{(s_j,p_j): j \geq 1\}$ is an interpolating sequence, given any bounded sequence $w=\{w_1, w_2, \ldots\}$, we need to find a constant $C_w$ such that  for every admissible kernel $k$,
\begin{eqnarray}\label{auxilaryzero}
\left( \left( \; \; (C_w^2 - w_i \overline{w}_j) k(\lambda_i , \lambda_j) \; \; \right) \right) \geq 0,
\end{eqnarray}
which, by Lemma \ref{infinitePick}, will prove our assertion. For any integer $n \ge 1$, choosing
$$\widetilde{c_j} = \left\{ \begin{array}{c} c_jw_j\sqrt{k((s_j,p_j),(s_j,p_j))} \text{ if }  1 \le j \le n,\\
 0 \text{ if } j > n \end{array} \right. $$ in the second inequality of (\ref{theconverse}), we get
\begin{equation}\label{auxilaryone}
\sum_{i,j=1}^nc_i\overline{c_j}w_i\overline{w_j}k((s_i,p_i),(s_j,p_j))\leq M \big(\sup_{i\ge 1} |w_i|\big)^2 \sum_{j=1}^n|c_j|^2k((s_j,p_j),(s_j,p_j)).
\end{equation}
Choosing
$$\widetilde{c_j}' = \left\{ \begin{array}{c} c_j\sqrt{k((s_j,p_j),(s_j,p_j))} \text{ if }  1 \le j \le n,\\
 0 \text{ if } j > n \end{array} \right. $$
  in the first inequality of (\ref{theconverse}), we get
\begin{eqnarray}\label{auxilarytwo}
\sum_{j=1}^n|c_j|^2k((s_j,p_j),(s_j,p_j))\leq M\sum_{i,j=1}^nc_i\overline{c_j}k((s_i,p_i),(s_j,p_j)).
\end{eqnarray}
Combining (\ref{auxilaryone}) and (\ref{auxilarytwo}) we get
$$
\sum_{i,j=1}^nc_i\overline{c_j}w_i\overline{w_j}k((s_i,p_i),(s_j,p_j))\leq  M^2 \big(\sup_{i\ge 1} |w_i|\big)^2 \sum_{i,j=1}^nc_i\overline{c_j}k((s_i,p_i),(s_j,p_j)). $$
Now, for any $l^2$ sequence $(c_j)$, we have the inequality above for any $n\ge 1$. Putting $C_w = M \big(\sup_{i\ge 1} |w_i|\big)$, the required inequality (\ref{auxilaryzero}) follows.

We now prove that $(ii)$ is equivalent to $(iii)$. First observe that condition (\ref{grammbelow}) is equivalent to the following:
\begin{eqnarray}\label{schurproduct}
(N-\delta_{ij})\cdot k((s_i,p_i),(s_j,p_j)) \geq 0,
 \end{eqnarray}for every admissible kernel $k$ on $\mathbb{G}$. Let $Y=\{(s_j,p_j): j\geq 1\}$. By the scalar-valued version of the Realization Theorem described in Section 2 (i.e., the way it is stated in Subsection 1.3, page 510 of \cite{BS}), there is a completely positive function $\Delta: Y \times  Y \to C(\overline{\mathbb D})^*$ such that for every $i,j\geq 1$,
 $$
 N-\delta_{ij}=\Delta((s_i,p_i),(s_j,p_j))\big(1-\varphi(\cdot,s_i,p_i)\overline{\varphi(\cdot,s_j,p_j)}\big).
 $$
 Let $\{e_j:j\geq1\}$ be the canonical orthogonal basis of $l^2$. Rewriting the above term we get
 $$ N + \Delta ( (s_i,p_i),(s_j,p_j)) \big( \varphi(\cdot, s_i, p_i) \overline{\varphi(\cdot , s_j, p_j)} \big) = \langle e_i,e_j \rangle + \Delta ((s_i,p_i),(s_j,p_j)) \big( 1  \big).$$
 By Lemma \ref{crucial-lemma}, there is a Hilbert space $\mathcal H$ and a function $L : Y \times Y \rightarrow B(C(\overline{\mathbb D}) ,\mathcal H)$ such that
 $$\Delta \left((s_i,p_i),(s_j,p_j) \right) (h_1 \overline{h_2}) =
\langle L(s_i,p_i) h_1, L(s_j,p_j)h_2 \rangle_{\mathcal H}$$ for every $h_1, h_2$ in $C(\overline{\mathbb D})$. Hence
 $$ N +  \langle L (s_i,p_i) \varphi(\cdot , s_i, p_i), L (s_j,p_j) \varphi(\cdot , s_j, p_j) \rangle = \langle e_i,e_j \rangle + \langle L (s_i,p_i) 1 , L (s_j,p_j) 1 \rangle.$$
 By equation (\ref{star-represt.}), this is the same as
 $$ N +   \langle \pi \varphi(\cdot , s_i, p_i) L (s_i,p_i) 1,  \pi \varphi(\cdot , s_j, p_j) L (s_j,p_j) 1 \rangle = \langle e_i,e_j \rangle + \langle L (s_i,p_i) 1 , L (s_j,p_j) 1 \rangle.$$
 Now we can define an isometry $V$ from the span of $$\{\sqrt{N} \oplus \pi \varphi(\cdot , s_j, p_j) L (s_j,p_j) 1 : j\geq1\}\subseteq \mathbb C\oplus \mathcal H$$ into the span of $\{e_j \oplus L (s_j,p_j) 1: j\geq1\}\subseteq l^2\oplus \mathcal H$ such that for each $ j\geq1$
 $$ V \left( \begin{array}{c} \sqrt N \\ \pi \varphi(\cdot , s_j, p_j) L (s_j,p_j) 1 \end{array} \right) = \left( \begin{array}{c} e_j \\ L (s_j,p_j) 1 \end{array} \right)$$
 and then extending by linearity. By a standard technique of adding an infinite dimensional Hilbert space to $\mathcal H$, if required, we can extend $V$ to a unitary from $\mathbb C \oplus \mathcal H$ onto $l^2 \oplus \mathcal H$. Now, write $V$ as
 \begin{eqnarray}\label{actionofV}
 \left(
     \begin{array}{cc}
       A & B \\
       C & D \\
     \end{array}
   \right)
   \end{eqnarray} and define $F:\mathbb G\to \mathcal B(\mathbb C,l^2)$ as
   $$ F(s,p)  =  A + B \pi \varphi(\cdot , s, p) (I_H - D \pi \varphi(\cdot , s, p) )^{-1} C . $$ By the Realization Theorem, $F$ is a bounded analytic function and $\|F\|_\infty\leq 1$. By (\ref{actionofV}) we have
 \begin{eqnarray}
 A \sqrt{N} + B \pi \varphi(\cdot , s_j, p_j) L (s_j, p_j) 1 & = & e_j \mbox{ and} \label{formulaforphi} \\
 C \sqrt{N} + D \pi \varphi(\cdot , s_j, p_j) L (s_j, p_j) 1 & = & L(s_j, p_j) 1. \nonumber \end{eqnarray}
Eliminating $L(s,p) 1$ we get that the function $\Phi:\mathbb G\to \mathcal B(\mathbb C,l^2)$ defined by $\Phi=\sqrt N F$ has the property $\Phi(s_j,p_j)=e_j$ for each $j \geq 1$. Therefore we have proved the following lemma.
 \begin{lem}\label{boundedbelow}
 Condition (\ref{grammbelow}) is equivalent to the existence of a function $\Phi$ in $H^\infty(\mathbb{G},\mathcal{B}(\mathbb{C}, l^2))$ of norm at most $\sqrt{N}$ such that $\Phi(s_j,p_j)=e_j$ for each $j \geq 1$.
 \end{lem}
 Also note that condition (\ref{grammabove}) is equivalent to
 $$
 (M\delta_{ij}-1)k((s_i,p_i),(s_j,p_j)) \geq 0.
 $$ Proceeding as above one gets the following result.
 \begin{lem}\label{boundedabove}
 Condition (\ref{grammabove}) is equivalent to the existence of a function $\Psi$ in $H^\infty(\mathbb{G},\mathcal{B}( l^2, \mathbb{C}))$ of norm at most $\sqrt{M}$ such that $\Psi(s_j,p_j)e_j=1$ for each $j \geq 1$.
 \end{lem}
 Now we are ready to prove that $(ii)$ and $(iii)$ are equivalent. Suppose $(ii)$ holds. Then by Lemma \ref{boundedbelow} there exists a function $\Phi$ in $H^\infty(\mathbb{G},\mathcal{B}(\mathbb{C}, l^2))$ of norm at most $\sqrt{N}$ such that $\Phi(s_j,p_j)=e_j$ for each $j \geq 1$. Let $\phi_1,\phi_2,\dots$ be functions on $\mathbb{G}$ such that
 $$
 \Phi(s,p)=(\phi_1(s,p),\phi_2(s,p),\dots)^t. $$
 The norm of $\Phi$ is no greater than $\sqrt{N}$. This implies that for all $(s,p)\in \mathbb{G}$,
 $$\sum_i|\phi_i(s,p)|^2 \leq N.$$
 Define $\Psi(s,p) = \Phi(s,p)^t$. Then $\Psi$ is a function in $H^\infty(\mathbb{G},\mathcal{B}( l^2, \mathbb{C}))$ of norm at most $\sqrt{N}$ such that $\Psi(s_i,p_i)e_i=1$ for each $i$ and hence by Lemma \ref{boundedabove} condition (\ref{grammabove}) holds with the constant $N$ in place of $M$. Note that for each $i \geq 1$, $\phi_i$ is the function such that $\phi_i(s_j,p_j)=\delta_{ij}$, for all $j \geq 1$. Hence $\{(s_j,p_j): j\geq 1\}$ is strongly separated.

 Conversely, suppose $(iii)$ holds. Therefore by Lemma \ref{boundedabove} there exists a function $\Psi$ in $H^\infty(\mathbb{G},\mathcal{B}( l^2, \mathbb{C}))$ of norm at most $\sqrt{M}$ such that $\Psi(s_i,p_i)e_i=1$ for each $i \geq 1$. Write $\Psi$ as
 $$
 \Psi(s,p)=(\psi_1(s,p),\psi_2(s,p),\dots),
 $$where the functions $\psi_i$s are such that $\sum_i|\psi_i(s,p)|^2 \leq M$ and $\psi_i(s_i,p_i)=1$ for each $i$. Moreover, the sequence $\{(s_j,p_j): j\geq 1\}$ is strongly separated. This means there exist a constant $L$ and a sequence $\varphi_i$ of functions on $\mathbb{G}$ such that $\varphi_i(s_j,p_j)=\delta_{ij}$ for each $j$ and $\|\varphi_i\|_\infty \leq L$. Define
 $$
 \Phi(s,p)=(\varphi_1(s,p)\psi_1(s,p),\varphi_2(s,p)\psi_2(s,p),\dots)^t.
 $$ Clearly, $\|\Phi\|_\infty \leq L\sqrt{M}$ and $\Phi(s_i,p_i)=e_i$ for all $i$ which, by Lemma \ref{boundedbelow}, proves that $(ii)$ holds.

 Note that $(i)$  and  $(iv)$ together are equivalent to $(ii)$ and $(iii)$ together. We have proved that $(i)$  is equivalent to  $(iv)$, and $(ii)$ is equivalent to  $(iii)$. Hence the proof of Theorem \ref{Interpolate} is complete. \qed

We end this section with a sufficient condition for a sequence to be interpolating. Suppose $\{(s_j,p_j): j \geq 1\}$ be a sequence of points in $\mathbb{G}$ such that for some $\alpha$ in $\overline{\mathbb{D}}$, the sequence $\{z_j=\varphi(\alpha,s_j,p_j): j\geq 1\}$ in $\mathbb{D}$ is interpolating, where $\varphi(\alpha, \cdot)$ is the coordinate function as defined by (\ref{testfunction}). Then the sequence $\{(s_j,p_j): j \geq 1\}$ is also interpolating. Because for each bounded sequence $w=\{w_j: j \geq 1\}$, the function $g_w\circ \varphi ( \alpha, \cdot )$ interpolates $(s_j,p_j)$ to $w_j$, where $g_w$ is the function from $\mathbb D$ to $\mathbb D$ that interpolates $z_j$ to $w_j$.
%It is the converse part that seems to be challenging. Consider the following diagram:
%\[
%\begin{tikzcd}
%{} & \mathbb{D} \arrow{dr}{?} \\
%\mathbb{G} \arrow{ur}{f_w} \arrow{rr}{\varphi_\alpha} && \mathbb{D}
%\end{tikzcd}
%\]
%So, one way of proving the converse part is by `finding a suitable inverse of the test function'. I do not know if that is even possible. I have to read the papers of A-Y, where they talked about the test functions, for more properties of it. Or, we may have to use the known characterizations of interpolating sequences.
So we have the following Carleson-type condition.
\begin{lem}\label{suff}
Let $\{(s_j,p_j): j \geq 1\}$ be a sequence of points in $\mathbb{G}$. Let $\alpha$ be in $\overline{\mathbb{D}}$ and $\varphi(\alpha,\cdot)$ be the coordinate function as defined in (\ref{testfunction}). Denote $z_j:=\varphi(\alpha,s_j,p_j)$. If there exits $\delta>0$ such that
$$
\prod_{j\neq k}\left|\frac{z_j-z_k}{1-\overline{z_k}z_j}\right| \geq \delta, \text{ for all }k,
$$ then $\{(s_j,p_j): j \geq 1\}$ is an interpolating sequence.
\end{lem}

\section{Cyclic $\Gamma$--isometries}

This short section proves a result about cyclic $\Gamma$--isometries. The main result of this section will be used in the proof of Theorem 2.

A $\Gamma$--contraction $(R,U)$ is called a $\Gamma$--{\em unitary}
if $U$ is a unitary operator. In such a case, $R$ and $U$ are normal operators and the joint spectrum
$\sigma(R,U)$ of $(R,U)$ is contained in the distinguished
boundary of $\Gamma$.

A $\Gamma$--contraction $(T,V)$ acting on a Hilbert space $\mathcal K$ is called a $\Gamma$--{\em isometry} if there
exist a Hilbert space $\mathcal{N}$ containing $\mathcal{K}$ and a
$\Gamma$--unitary $(R,U)$ on $\mathcal{N}$ such
that $\mathcal{K}$ is left invariant by both $R$ and
$U$, and
$$T = R|_{\mathcal{K}} \mbox{ and } V = U|_{\mathcal{K}}.$$

In other words, $(T,V)$ is a $\Gamma$--isometry if it has a $\Gamma$--unitary
extension $(R,U)$.

The two following theorems are from \cite{ay-jot} and characterize $\Gamma$--unitaries and $\Gamma$--isometries.

\begin{thm} \label{G-unitary}
Let $(R,U)$ be a pair of commuting operators defined on a Hilbert
space $\mathcal{H}.$ Then the following are equivalent:
\begin{enumerate}

\item $(R,U)$ is a $\Gamma$--unitary;
\item there exist commuting unitary operators $U_{1}$ and $
    U_{2}$ on $\mathcal{H}$ such that
$$R= U_{1}+U_{2},\quad U= U_{1}U_{2};$$
\item $U$ is unitary,\;$R=R^*U,\;$\;and $\| R \| \leq2$.
    \item $U$ is a unitary and $R = W + W^* U$ for some
        unitary $W$ commuting with $U$.
\end{enumerate}
\end{thm}

\begin{thm} \label{G-isometry}
Let $T\,,\,V$ be commuting operators on a Hilbert space
$\mathcal{H}.$ The following statements are all
equivalent:
\begin{enumerate}
\item $(T,V)$ is a $\Gamma$--isometry,
\item $(T,V)$ is a $\Gamma$--contraction and $V$ is isometry,
\item $V$ is an isometry\;,\;$T=T^*V$ and $\| T \| \leq2.$

\end{enumerate}
\end{thm}
A $\Gamma$--coisometry is the adjoint (componentwise) of a $\Gamma$--isometry. Agler and Young proved the following remarkable result which we shall need.

\begin{thm}[Agler and Young, Theorem 3.1 in \cite{ay-jot}] \label{dilation} Let $(S,P)$ be a $\Gamma$--contraction on a Hilbert space $\mathcal H$. There exists a Hilbert space $K$ containing $\mathcal H$ and a $\Gamma$--coisometry $(S^\flat, P^\flat)$ on $K$ such that $\mathcal H$ is invariant under $S^\flat$ and $P^\flat$, and $S = S^\flat|_H, P = P^\flat|_H$. \end{thm}

Let $\mu$ be a regular Borel measure on $b\Gamma$. On $H^2(b\Gamma , \mu)$, define two commuting bounded operators
$$(M_s^\mu f)(s,p) = sf(s,p) \mbox{ and } M_p^\mu f(s,p) = pf(s,p).$$
Since $(s,p) \in b\Gamma$ (equivalently, $|p|=1, s = \overline{s}p$ and $|s|\le 2$), it is easy to check that $(M_s^\mu ,M_p^\mu )$ is a $\Gamma$--isometry on $H^2(b\Gamma , \mu)$. Indeed, according to one of the characterizations of a $\Gamma$--isometry given above, we need to show that $M_p^\mu $ is an isometry, $M_s^\mu = (M_s^\mu)^*M_p^\mu$ and $\|M_s^\mu\| \le 2$ all of which follow from the fact that $(s,p) \in b\Gamma$. Moreover, $(M_s^\mu ,M_p^\mu )$ is $cyclic$ with the constant function $1$ serving as the cyclic vector because
$$\overline{{\rm span}} \{(M_s^\mu)^m(M_p^\mu)^n 1 : m,n \in \mathbb N\} = H^2(b\Gamma, \mu).$$

Conversely, if $(T,V)$ is a $\Gamma$--isometry on $\mathcal H$ with a cyclic vector $h_0$, we extend it to a $\Gamma$--unitary $(R,U)$ on $K$, say. Since $R$ and $U$ are commuting normal operators, the $C^*$-algebra $C^*(R,U)$ generated by them is commutative. The closure of the subspace $\{Xh_0: X \in C^*(R,U)\}$ is a reducing subspace of $(R,U)$ and contains $\mathcal H$ as an invariant subspace of $(R,U)$. So, we can, without loss of generality, assume $K$ to be the above space. Hence, $(R,U)$ is a $minimal$ dilation. The $\Gamma$--unitary $(R,U)$ is cyclic too (i.e., $K = \{Xh_0: X \in C^*(R,U)\}$) with the same cyclic vector $h_0$. Applying Gelfand theory to $C^*(R,U)$ and remembering that the joint spectrum of $(R,U)$ is contained in $b\Gamma$, we get a measure $\mu$ on $b\Gamma$ such that $(R,U)$ is unitarily equivalent to $(M_s^\mu,M_p^\mu)$ on $L^2(b\Gamma, \mu)$ and the $\Gamma$--isometry $(T,V)$ is the restriction of $(M_s^\mu,M_p^\mu)$ to $H^2(b\Gamma, \mu)$. Summing up, we have proved the following.

\begin{lem}\label{cyclic} A commuting pair of bounded operators $(T,V)$ is a cyclic $\Gamma$--isometry if and only if there is a regular Borel measure $\mu$ on $b\Gamma$ such that $(T,V)$ is unitarily equivalent to $(M_s^\mu,M_p^\mu)$ on $H^2(b\Gamma, \mu)$. \end{lem}

\section{The Toeplitz corona theorem on the symmetrized bidisk -- Proof of Theorem \ref{TC-G}}

The following lemma plays a pivotal role in the proof of Theorem \ref{TC-G}.
\begin{lem}\label{factor}
Let $Y$ be a subset of $\mathbb G$ and $J:Y\times Y\to \mathcal{B}(\mathcal{L})$ be a continuous self-adjoint (i.e., $J((s,p),(t,q))=J((t,q),(s,p))^*$) function. If
\begin{eqnarray}\label{weaker assumption}
J \oslash k : \big((s,p), (t,q)\big) \mapsto J\big((s,p), (t,q)\big) \otimes k\big((s,p), (t,q)\big)
\end{eqnarray}is positive semi-definite for every $\mathcal{B}(\mathcal{L})$-valued admissible weak kernel $k$, then there is a completely positive function $\Delta: Y \times  Y \to \mathcal B(C(\overline{\mathbb D}), \mathcal L)$ such that for every $(s,p),(t,q)$ in $Y$,
\begin{eqnarray*}
J\big((s,p), (t,q)\big) = \Delta((s,p),(t,q))\big(1-\varphi(\cdot,s,p)\overline{\varphi(\cdot,t,q)}\big).
\end{eqnarray*}
\end{lem}
\begin{proof}
We first prove the result for finite subsets of $Y$ and then apply Kurosh's theorem. Let $\mathcal F=\{(s_j,p_j):1\leq j\leq N\}$ be a finite subset of $Y$ of cardinality $N$. Consider the following subset of $N \times N$ self-adjoint operator matrices with entries in $\mathcal{B}(\mathcal{L})$,
\begin{align*}
\mathcal{W} & =
\big\{\left[\Delta\big((s_i,p_i),(s_j,p_j)\big)\left(1-\varphi(\cdot,s_i,p_i)\overline{\varphi(\cdot,s_j,p_j)}\right)\right]_{i,j=1}^N:\\&\Delta: \mathcal F \times \mathcal F \to \mathcal{B}\left(C(\overline{\mathbb D}), \mathcal{B}(\mathcal L)\right) \text {completely positive function}\big\}.
\end{align*}
The subset $\mathcal{W}$ of $\mathcal{B}(\mathcal{L}^N)$ is a wedge in the vector space of $N \times N$ self-adjoint matrices with entries from $\mathcal B(\mathcal L)$ in the sense that it is convex and if we multiply a member of $\mathcal{W}$ by a non-negative real number, then the element remains in $\mathcal{W}$. Since $\mathcal{B}(\mathcal{L}^N)$ is the dual of $\mathcal{B}_1(\mathcal{L}^N)$, the ideal of trace class operators acting on $\mathcal{L}^N$, it has its natural weak-star topology. We shall show that it is closed. This will require some work. We shall pick up the proof of the lemma after we show that the wedge is closed. Let
$$
K_\nu\big((s_i,p_i),(s_j,p_j)\big)=\Delta_\nu\big((s_i,p_i),(s_j,p_j)\big)\left(1-\varphi(\cdot,s_i,p_i)\overline{\varphi(\cdot,s_j,p_j)}\right)
$$
be a net in $\mathcal W$ which is indexed by $\nu$ in some index set and which converges to an $N \times N$ self-adjoint $\mathcal B(\mathcal L)$-valued matrix $K=(K_{ij})$ with respect to the weak-star topology. This means that for every $X=(X_{kl})\in \mathcal{B}_1(\mathcal{L}^N)$, the net of scalars $\tr(K_{\nu}X)$ converges to $\tr(KX)$. Let us use a special $X$. Consider two vectors $u$ and $v$ in $\mathcal{L}$ and choose $X$ to be the block operator matrix which has $u\otimes v$ in the $(ji)$-th entry and zeroes elsewhere. Then we get
$$
\langle \Delta_\nu\big((s_i,p_i),(s_j,p_j)\big)\left(1-\varphi(\cdot,s_i,p_i)\overline{\varphi(\cdot,s_j,p_j)}\right)u,v \rangle \to \langle K_{ij}u, v\rangle
$$
for all $i=1,2, \ldots , N$ and all $j=1,2, \ldots , N$. In particular, we have
$$
\langle \Delta_\nu\big((s_i,p_i),(s_i,p_i)\big)\left(1-|\varphi(\cdot,s_i,p_i)|^2\right)u,u \rangle \to \langle K_{ii}u, u\rangle
$$for all $u\in \mathcal L$ and $1\leq i \leq N$.  Let us recall that for any $(s,p) \in \mathbb G$, $\sup\{ | \varphi (\alpha , s, p) | : \alpha \in \overline{\mathbb D} \} < 1$, so that we have an $\epsilon >0$ satisfying $1-|\varphi(\cdot,s_i,p_i)|^2\geq \epsilon 1$ for each $1\leq i \leq N$. Hence for each $1\leq i \leq N$,
$$
\langle \Delta_\nu\big((s_i,p_i),(s_i,p_i)\big)\left(1-|\varphi(\cdot,s_i,p_i)|^2\right)u,u \rangle \geq \epsilon \langle \Delta_\nu\big((s_i,p_i),(s_i,p_i)\big)(1)u,u \rangle.
$$
Since the left hand side of the above inequality converges for every $i$ and there are only finitely many $i$'s, we have a positive constant $M(u)$ depending only on $u$ such that
$$
\sup_\nu \langle \Delta_\nu \big((s_i,p_i),(s_i,p_i)\big)(1)u,u \rangle <M(u).
$$
 Since we have $\|h\|_\infty^2-|h(\cdot)|^2\geq 0$ for every $h \in C(\overline{\mathbb D})$, hence
\begin{eqnarray*}
\langle \Delta_\nu\big((s_i,p_i),(s_i,p_i)\big)(|h|^2)u,u \rangle\leq \|h\|_\infty^2\langle \Delta_\nu\big((s_i,p_i),(s_i,p_i)\big)(1)u,u \rangle \leq \|h\|_\infty^2M(u).
\end{eqnarray*} For a completely positive function $\Delta$, we have for every $h_1,h_2 \in C(\overline{\mathbb D})$ and $u,v \in \mathcal L$,
$$
|\langle \Delta\big( (s,p), (t,q) \big)(h_1\overline{h_2})u,v\rangle|
\leq\langle \Delta\big( (s,p), (s,p) \big)(|h_1|^2)u,u\rangle\langle \Delta\big( (t,q), (t,q) \big)(|h_2|^2)v,v\rangle,
$$
which immediately gives a bound on off-diagonal entries
\begin{eqnarray*}
|\langle \Delta_\nu\big((s_i,p_i),(s_j,p_j)\big)(h)u,v \rangle|\leq \|h\|_\infty^2M(u)M(v),
\end{eqnarray*}for every $h \in C(\overline{\mathbb D})$, all $u,v \in \mathcal L$ and every $\nu$. Therefore, for every $h \in C(\overline{\mathbb D})$ and $u,v \in \mathcal L$, the net $\{\langle \Delta_\nu\big((s_i,p_i),(s_j,p_j)\big)(h)u,v \rangle\}$ is bounded, for each $1\leq i,j \leq N$. Since the set $\mathcal F$ is finite, we get a subnet $\nu_l$ such that $\{\langle \Delta_{\nu_l}\big((s_i,p_i),(s_j,p_j)\big)(h)u,v \rangle\}$ converges to some complex number (depending on $i,j,h,u$ and $v$). Now we define a completely positive function
$\Delta: \mathcal F \times \mathcal F \to \mathcal{B}\left(C(\overline{\mathbb D}), \mathcal{B}(\mathcal L)\right)$ by
$$
\langle \Delta\big( (s_i,p_i), (s_j,p_j) \big)(h)u,v \rangle = \lim_l\langle \Delta_{\nu_l}\big( (s_i,p_i), (s_j,p_j) \big)(h)u,v \rangle
$$ and extend it trivially to $Y \times Y$. Consequently, for every $h \in C(\overline{\mathbb D})$ and $u,v \in \mathcal L$,
$$
\langle \Delta\big( (s_i,p_i), (s_j,p_j) \big)(h)u,v \rangle =\langle K_{ij}u,v \rangle \text{ for each }1\leq i,j\leq N
$$ proving that $\mathcal{W}$ is weak-star closed and hence operator norm closed too.

 Continuing the proof of the lemma, for a function $\delta:\overline{\mathbb D}\times \mathbb G\times \mathbb G\to \mathcal B(\mathcal L)$ such that for each $\alpha \in \overline{\mathbb D}$, $\delta(\alpha, \cdot, \cdot)$ is a weak kernel on $\mathbb G$, let $\Delta^\delta_\mu$ be as defined in (\ref{ExampleofDelta}). Consider the two functions $b,d:\overline{\mathbb D}\times \mathbb G\times \mathbb G\to \mathcal B(\mathcal L)$ defined by
\begin{eqnarray}\label{the-b-kernel}
b(\alpha,(s,p),(t,q))=\frac{I_\mathcal L}{1-\varphi(\alpha,s,p)\overline{\varphi(\alpha,t,q)}},
\end{eqnarray}
and
\begin{eqnarray}\label{the-d-kernel}
d(\alpha,(s,p),(t,q))=\frac{[u(s,p) \otimes u(t,q)]}{1-\varphi(\alpha,s,p)\overline{\varphi(\alpha,t,q)}},
\end{eqnarray}where $\bf u:\mathbb G\to \mathcal L$ is a function and for two elements $u_1,u_2$ of $\mathcal L$, $(u_1\otimes u_2)$ denotes the bounded operator on $\mathcal L$ defined by
$$
(u_1\otimes u_2)(h)=\langle h,u_2\rangle u_1.
$$
Then for a probability measure $\mu$ on $\overline{\mathbb D}$, we have
$$\Delta^b_\mu\big( (s,p), (t,q) \big)\big(1-\varphi(\cdot,s,p)\overline{\varphi(\cdot,t,q)}\big)=I_\mathcal L \text{ for every $(s,p),(t,q)\in \mathbb G$}$$and
$$\Delta^d_\mu\big( (s,p), (t,q) \big)\big(1-\varphi(\cdot,s,p)\overline{\varphi(\cdot,t,q)}\big)={\bf u}(s,p)\otimes {\bf u}(t,q)\text{ for every }(s,p),(t,q)\in \mathbb G$$
 and hence we conclude that the block operator matrix with each entry being $I_{\mathcal L}$ is in $\mathcal W$ and if $u_1,u_2,\dots,u_N$ are any vectors in $\mathcal L$, then the $N \times N$ matrix
$$
D(i,j)=u_i\otimes u_j \text{ for each } 1\leq i,j \leq N
$$
is also in $\mathcal{W}$.

We now show that the restriction $J_{\mathcal F}$ of $J$ to $\mathcal F \times \mathcal F$ is in $\mathcal{W}$. Suppose on the contrary that $J_{\mathcal F}$ is not in $\mathcal{W}$. Then it is well-known as a consequence of the Hahn--Banach extension theorem that these two can be separated by a weak-star continuous linear functional $L$ on $\mathcal{B}(\mathcal{L}^N)$. Specifically, applying part (b) of Theorem 3.4 of \cite{Rud}, we get such an $L$ whose real part is non-negative on $\mathcal{W}$ and strictly negative on $J_{\mathcal F}$. We replace this linear functional by its real part, i.e., $\frac{1}{2}(L(T)+\overline{L(T)})$ and denote it by $L$ itself. Thus, without loss of generality we can take $L$ to be real-valued.

Since $L$ is weak-star continuous, $L$ has a specific form. In fact, there is an $N \times N$ self-adjoint operator matrix $K$ with entries in the ideal of trace class operators such that
$$ L(T)=\tr(TK). $$
 This is also well-known and can be found for example in Theorem 1.3 of Chapter V of \cite{conway func}. Let us define $K^t$ by $K^t(\lambda_i, \lambda_j)=K(\lambda_j, \lambda_i)^t$. Let $\{e_n : n \in \mathbb{N}\}$ be an orthonormal basis for $\mathcal{L}$. For $u=\sum c_m e_m$ and $v=\sum d_n e_n$ in $\mathcal{L}$, we make a note of the following fact about $K^t$, which will be used later in the proof.
\begin{eqnarray*}
\langle K^t(\lambda_i, \lambda_j)u,v \rangle&=&\sum_{m, n}c_m \bar{d_{n}}\langle K(\lambda_j, \lambda_i)^te_m,e_n \rangle\\
&=&\sum_{m,n}c_m \bar{d_{n}}\langle K(\lambda_j, \lambda_i)e_n,e_m \rangle = \langle K(\lambda_j, \lambda_i)\bar{v},\bar{u} \rangle,
\end{eqnarray*}
where $\bar{u}=\sum \bar{c}_m e_m$ and $\bar{v}=\sum \bar{d_n} e_n$.

It is simple to show that $K^t$ is a $\mathcal{B}(\mathcal{L})$-valued positive semi-definite kernel on $\mathcal F$, i.e.,
\begin{eqnarray}\label{action on D}
\sum_{i,j=1}^N\langle K^t(\lambda_i, \lambda_j)u_j,u_i \rangle \geq 0,
\end{eqnarray} where $u_1,u_2,\dots,u_N$ are arbitrary vectors in $\mathcal{L}$. The following shows that (\ref{action on D}) is the action of $L$ on the kernel $D(i,j)=[\bar{u_i} \otimes \bar{u_j}]$ and hence we are done.
\begin{eqnarray*}
0 \leq L(D)= \tr(D K)=\sum_{i,j=1}^N\tr(D_{ij}K_{ji})&=&\sum_{i,j=1}^N\tr([\bar{u_i} \otimes K_{ji}^*\bar{u_j}])=\sum_{i,j=1}^N\langle \bar{u_i}, K_{ji}^*\bar{u_j} \rangle \\&=& \sum_{i,j=1}^N \langle K_{ji}\bar{u_i}, \bar{u_j} \rangle = \sum_{i,j=1}^N \langle K^t(\lambda_i, \lambda_j)u_j,u_i \rangle.
\end{eqnarray*}
The next step is to show that $K^t$ is admissible. Lemma \ref{admi} will be used now. This is a matter of choosing the completely positive function judiciously. Note that for each $\alpha \in \overline{\mathbb D}$ and a function $u:\mathbb G \to \mathcal L$, the function $\Delta^\alpha: \mathbb G \times \mathbb G \to \mathcal{B}\left(C(\overline{\mathbb D}), \mathcal{B}(\mathcal L)\right)$ defined by
$$
\Delta^\alpha\big( (s,p), (t,q) \big)(h)=h(\alpha)[u(s,p)\otimes u(t,q)]
$$is completely positive, because $\Delta^\alpha=\Delta^\delta_\mu$, as defined in (\ref{ExampleofDelta}) with $\delta(\cdot,(s,p),(t,q))=u(s,p)\otimes u(t,q)$ and $\mu$ being the point mass measure at $\alpha$. This implies that for each $\alpha\in \overline{\mathbb D}$ and vectors $u_1,u_2,\dots,u_N$ in $\mathcal L$, the following $\mathcal B(\mathcal L)$-valued $N\times N$ matrix
$$
A(\alpha)=\left(\left(\big(1-\varphi(\alpha,s_i,p_i)\overline{\varphi(\alpha,s_j,p_j)}\big)[u_i\otimes u_j]\right)\right)_{i,j=1}^N
$$is in $\mathcal{W}$. The fact that $L$ is non-negative on $\mathcal{W}$ shows that $K^t$ is admissible.

Therefore by hypothesis, the $\mathcal{B}(\mathcal{L} \otimes \mathcal{L})$-valued function $J_{\mathcal F} \oslash K^t$ on $Y \times Y$ is positive semi-definite, which means that for every choice of vectors $\{u_i\}_{i=1}^N$ in $\mathcal{L} \otimes \mathcal{L}$,  we have
\begin{eqnarray}\label{fun}
\sum_{i,j=1}^N \langle J_{\mathcal F} \oslash K^t(\lambda_i, \lambda_j)u_j, u_i\rangle \geq 0.
\end{eqnarray} For a finite subset $\mathcal{F}=\{1,2,\dots,R\}$ of $\mathbb{N}$, choose $u_i=\sum_{m =1}^R e_m \otimes e_m$ for each $i$. Note that for this choice of $u_i$, (\ref{fun}) is the same as
\begin{eqnarray}\label{fun1}
\sum_{i,j=1}^N\sum_{m, n =1}^R \langle J_{\mathcal F} (\lambda_i, \lambda_j)e_m, e_n \rangle\langle K^t(\lambda_i, \lambda_j)e_m, e_n \rangle \geq 0.
\end{eqnarray}
On the other hand,
\begin{eqnarray*}
L(J_{\mathcal F})&=&\sum_{i,j=1}^N\tr(J_{\mathcal F}(\lambda_i, \lambda_j)K(\lambda_j, \lambda_i))
\\
&=&
\sum_{i,j=1}^N\sum_{n =1}^\infty\langle J_{\mathcal F}(\lambda_i, \lambda_j)K(\lambda_j, \lambda_i)e_n, e_n\rangle\\
&=&
\sum_{i,j=1}^N\sum_{m, n=1}^\infty \langle J_{\mathcal F} (\lambda_i, \lambda_j)e_m, e_n \rangle\langle K(\lambda_j, \lambda_i)e_n, e_m \rangle \\
&=&\sum_{i,j=1}^N\sum_{m, n=1}^\infty  \langle J_{\mathcal F} (\lambda_i, \lambda_j)e_m, e_n \rangle\langle K^t(\lambda_i, \lambda_j)e_m, e_n \rangle \geq 0.
\end{eqnarray*}
 the last inequality following from (\ref{fun1}). Now $L(J_{\mathcal F})$ being non-negative is a contradiction to the assumption that $J_{\mathcal F}$ not in $\mathcal{W}$. Therefore the restriction $J_{\mathcal F}$ of $J$ to every finite subset $\mathcal F$ of $Y$ must be in $\mathcal{W}$. Now an application of Kurosh's theorem finishes the proof.
 \end{proof}

%here is an immediate consequence which we state below. We skip the proof as it is a standard result in $\mathcal C^*$--algebra theory, see \cite{BBLS}.

We shall actually prove the following general theorem from which the Toeplitz corona theorem for the symmetrized bidisk follows.
\begin{theorem}\label{gentpsymmbiD}
Let $\mathcal{L}_1, \mathcal{L}_2$ and $\mathcal{L}_3$ be Hilbert spaces and $Y$ be a subset of $\mathbb G$. Suppose $\Phi : Y\to \mathcal{B}(\mathcal{L}_1,\mathcal{L}_2)$ and $\Theta:Y\to \mathcal{B}(\mathcal{L}_3,\mathcal{L}_2)$ are given functions. Then the following statements are equivalent:
\begin{enumerate}
\item[(1)] There exists a function $\Psi$ in the closed unit ball of $H^\infty\big(\mathbb{G},\mathcal{B}(\mathcal{L}_3,\mathcal{L}_1)\big)$ such that
$$
\Phi(s,p)\Psi(s,p)=\Theta(s,p)
$$ for all $(s,p) \in Y$;
\item[(2)] The function
$$[\Phi(s,p)\Phi(t,q)^*-\Theta(s,p)\Theta(t,q)^*]\oslash k\big((s,p),(t,q)\big)$$ is positive semi-definite on $Y$ for every $\mathcal{B}(\mathcal{L}_2)$-valued admissible kernel $k$ on $Y$;
%\item[(2$'$)]The function
%$$[\Phi(\gamma(z))\Phi(\gamma(w))^*-\Theta(\gamma(z))\Theta(\gamma(w))^*]\oslash k\big(z,w\big)$$ is positive semi-definite on $\gamma^{-1}(Y)$ for every $\mathcal{B}(\mathcal{L}_2)$--valued admissible kernel $k$ on $\gamma^{-1}(Y)$;
\item[(3)] There exists a completely positive function $\Delta:Y \times Y \to \mathcal B\big(C(\overline{\mathbb D}), \mathcal B(\mathcal L_2)\big)$ such that for all $(s,p)$, $(t,q)$ in $Y$,
\begin{eqnarray*}
\Phi(s,p)\Phi(t,q)^*-\Theta(s,p)\Theta(t,q)^*&=&\Delta ( (s,p), (t,q)) \big( 1 - \varphi(\cdot, s, p) \overline{\varphi(\cdot , t, q)} \big).
\end{eqnarray*}
%\item[(3$'$)] There exist $\mathcal B(\mathcal L_2)$--valued weak kernels $K_1$ and $K_2$ on $\gamma^{-1}(Y)$ such that
%$$
%\Phi(\gamma(z))\Phi(\gamma(w))^*-\Theta(\gamma(z))\Theta(\gamma(w))^*=(1-z_1\overline{w_1})K_1(z,w)+(1-z_2\overline{w_2})K_2(z,w),
%$$for $z=(z_1,z_2)$ and $w=(w_1,w_2)$ in $\gamma^{-1}(Y)$.
\end{enumerate}
\end{theorem}
\begin{proof}
%We first establish $(1)\Leftrightarrow(2')\Leftrightarrow(3')$. Note that any function on a subset of the symmetrized bidisk gives rise to a symmetric function on the pre-image of that subset under $\gamma$ and conversely, any symmetric function on a subset of the bidisk gives rise to a function on the image of that subset under $\gamma$. Also, observe that if $\Psi$ is a function on a subset $X$ of the bidisk that satisfies
%$$
%\Phi(\gamma(z_1,z_2))\Psi(z_1,z_2)=\Theta(\gamma(z_1,z_2))
%$$
%for all $(z_1,z_2)$ in $X$, then $\Psi$ is symmetric on $X$. Now the Toeplitz corona theorem on the bidisk, i.e., Theorem 11.65 of \cite{ag-Mc} completes the proof.
%
%Now we prove $(1)\Leftrightarrow(2)\Leftrightarrow(3)$.
The proof will require the technique of the proof of the Realization Theorem.

 $(1)\Rightarrow(2):$  Suppose $(1)$ holds.
Since $\Psi$ is in the closed unit ball of $H^\infty\big(\mathbb{G},\mathcal{B}(\mathcal{L}_3,\mathcal{L}_2)\big)$, we apply the realization theorem with $Y=\mathbb G$ and $f=\Psi$ to get, by part $({\bf{M}})$ of The Realization Theorem,
$$
\left(I_{\mathcal{L}_2}-\Psi(s,p)\Psi(t,q)^*\right)\oslash k\big((s,p),(t,q)\big)
$$ is positive semi-definite for every $\mathcal{B}(\mathcal{L}_2)$-valued admissible kernel $k$ on $\mathbb{G}$.
 Now part $(2)$ follows from the following simple observation:
\begin{eqnarray*}
&&[\Phi(s,p)\Phi(t,q)^*-\Theta(s,p)\Theta(t,q)^*]\oslash k\big((s,p),(t,q)\big) \\
&=& \Phi(s,p)\big(I_{\mathcal{L}_1}-\Psi(s,p)\Psi(t,q)^*\big)\Phi(t,q)^*\oslash k\big((s,p),(t,q)\big).
\end{eqnarray*}

$(2)\Rightarrow(3):$ This is Lemma \ref{factor}.

$(3)\Rightarrow(1):$ This part of the proof uses a lurking isometry argument to construct the function $\Psi$. Suppose there exists a completely positive function $\Delta:Y \times Y \to \mathcal B\big(C(\overline{\mathbb D}), \mathcal B(\mathcal L)\big)$ such that for every $(s,p)$, $(t,q)$ in $Y$,
\begin{eqnarray*}
\Phi(s,p)\Phi(t,q)^*-\Theta(s,p)\Theta(t,q)^*&=&\Delta ( (s,p), (t,q)) \big( 1 - \varphi(\cdot, s, p) \overline{\varphi(\cdot , t, q)} \big) .
\end{eqnarray*}
We re-arrange the terms in the above equation and apply Lemma \ref{crucial-lemma} to obtain a Hilbert space $\mathcal H$, a function $L:Y\to B\big(C(\overline{\mathbb D}), \mathcal B(\mathcal H,\mathcal L_2)\big)$ and a unital $*$-representation $\pi:C(\overline{\mathbb D})\to \mathcal B(\mathcal H)$  such that
$$
\Phi(s,p)\Phi(t,q)^*+L(s,p)(\varphi(\cdot,s,p))L(t,q)(\varphi(\cdot,t,q))^*=\Theta(s,p)\Theta(t,q)^*+L(s,p)(1)L(t,q)(1)^*,
$$which implies that there exists an isometry $V_1$ from $$\overline{\text{span}}\{\Phi(t,q)^*e\oplus L(t,q)(\varphi(\cdot,t,q))^*e:(t,q)\in Y\; e\in \mathcal L_2\}\subset \mathcal{L}_1 \oplus \mathcal{H}$$ onto
$$\overline{\text{span}}\{\Theta(t,q)^*e\oplus L(t,q)(1)^*e :(t,q)\in Y\; e\in \mathcal L_2\}\subset \mathcal{L}_3 \oplus \mathcal{H}$$ such that for all $(t,q)\in Y$ and $e\in \mathcal L_2$,
\begin{eqnarray}\label{tpkaction}
\left(
   \begin{array}{c}
   \Phi(t,q)^* \\ \pi(\varphi(\cdot,t,q))^*L(t,q)(1)^*
   \end{array}
   \right)e \xrightarrow{V_1} \left(
   \begin{array}{c}
   \Theta(t,q)^* \\  L(t,q)(1)^*
   \end{array}
   \right)e.
 \end{eqnarray}
We add an infinite-dimensional summand to $\mathcal{H}$, if necessary, to extend $V_1$ as a unitary from $\mathcal{L}_1 \oplus \mathcal{H}$ onto $\mathcal{L}_3 \oplus \mathcal{H}$. Decompose $V_1$ as the $2\times 2$ block operator matrix
 $$\left(\begin{array}{cc}
       A_1 & B_1\\
       C_1 & D_1\\
     \end{array}
     \right)
 $$and define the function $\Psi$ on $\mathbb G$ by
 $$
 \Psi(t,q)^*=A_1+B_1\pi(\varphi(\cdot,t,q))^*(I_\mathcal H-D_1\pi(\varphi(\cdot,t,q))^*)^{-1}C_1.
$$ Then by the Realization Theorem $\Psi$ is a contractive multiplier and by (\ref{tpkaction}) it satisfies $\Psi(t,q)^*\Phi(t,q)^*=\Theta(t,q)^*$ for all $(t,q)$ in $Y$. Hence $(1)$ holds.
\end{proof}
\begin{proof}[\underline{Proof of Theorem \ref{TC-G}}:]
Note that equivalence of part (1) and (3) in Theorem \ref{TC-G} follows from Theorem \ref{gentpsymmbiD} when one chooses $Y=\mathbb G$ and $\Theta=\sqrt{\delta}$. We complete the proof of Theorem \ref{TC-G} by establishing $(1)\Rightarrow(2)\Rightarrow(3)$.

$(1)\Rightarrow(2):$ Denote the operator $(M^\mu_{{\psi_1}_r},M^\mu_{{\psi_2}_r},\dots,M^\mu_{{\psi_N}_r})^t$ by $M^\mu_{\psi_r}$. The inequality in part (1) shows that
$$M^{\mu *}_{\psi_r}M^\mu_{\psi_r}\leq \frac{1}{\delta} I_{H^2(b\Gamma,\mu)},$$which implies
$$
M^{\mu}_{\psi_r}{M^{\mu}_{\psi_r}}^*\leq\frac{1}{\delta}I_N,
$$which after conjugation by $(M^\mu_{{\varphi_1}_r},M^\mu_{{\varphi_2}_r},\dots,M^\mu_{{\varphi_N}_r})=:M^\mu_{\varphi_r}$ gives
$$
M^\mu_{\varphi_r} M^{\mu}_{\psi_r}{M^{\mu *}_{\psi_r}}{M^\mu_{\varphi_r}}^*\leq \frac{1}{\delta}M^\mu_{\varphi_r}{M^\mu_{\varphi_r}}^*,
$$whish establishes part (2), since $M^\mu_{\varphi_r} M^{\mu}_{\psi_r}=I_{H^2(b\Gamma,\mu)}$.

$(2)\Rightarrow(3):$ Let $k$ be an admissible $\mathcal B (\mathcal L_2)$-valued kernel on $\mathbb G$. As in \eqref{kernelspace}, we get a Hilbert space $H_k$ of $\mathcal L_2$-valued functions on $\mathbb G$. Define two operators $S$ and $P$ on $H_k$ by
$$Sf(s,p) = sf(s,p) \text{ and } Pf(s,p) = pf(s,p), \text{ where } (s,p) \in \mathbb G.$$
Since $k$ is admissible, the pair $(S,P)$ is a $\Gamma$--contraction. Hence, by Theorem \ref{dilation}, there is a $\Gamma$--coisometry $(S^\flat, P^\flat)$ which extends $(S,P)$. By assumption, we have
$$\Phi_r(M_s^\mu, M_p^\mu) \Phi_r(M_s^\mu, M_p^\mu)^* - \delta I \ge 0 \text{ for all } 0<r<1$$
for every measure $\mu$ on $b\Gamma$. By virtue of Lemma \ref{cyclic}, this means that
$$\Phi_r(T, V) \Phi_r(T, V)^* - \delta I \ge 0 \text{ for all } 0<r<1$$
and for any cyclic $\Gamma$--isometry $(T,V)$.

Now suppose $(T,V)$ is a $\Gamma$--isometry on $\mathcal H$ and $h \in \mathcal H$. Consider the subspace
$$ \mathcal M = \overline{\rm span} \{ T^mV^nh : m,n \ge 0\}. $$
This is an invariant subspace for $T$ and $V$. Let $T^\prime = T|_{\mathcal M}$ and $V^\prime = V|_{\mathcal M}$. Then, $(T^\prime , V^\prime)$ is a cyclic $\Gamma$--isometry. So, for all  $0<r<1$, we have
$$\langle \Phi_r(T, V) \Phi_r(T, V)^* h , h \rangle = \| \Phi_r(T, V)^* h \|^2 \ge \| P_{\mathcal M} \Phi_r(T, V)^* h \|^2 = \| (\Phi_r(T^\prime , V^\prime)^* h \|^2 \ge \delta \| h \|^2 $$
because of cyclicity of $(T^\prime , V^\prime)$. Thus we have
$$\Phi_r(T, V) \Phi_r(T, V)^* - \delta I \ge 0 \text{ for all } 0<r<1$$
and for any $\Gamma$--isometry $(T,V)$. Now, making use of Theorem \ref{dilation}, we get
$$\Phi_r(S, P) \Phi_r(S, P)^* - \delta I \ge 0 \text{ for all } 0<r<1.$$
This implies that $$ M_\Phi M_\Phi^* - \delta I = \Phi(S, P) \Phi(S, P)^* - \delta I \ge 0.$$ That is what was required to prove.
\end{proof}
\begin{rem}
We would like to conclude by noting that the proof of $(2)\Rightarrow(3)$ of Theorem \ref{TC-G} needs a characterization of cyclic $\Gamma$-isometries. Since a $\Gamma$-isometry in general cannot be obtained as a symmetrization of a pair of isometries (see \cite{ay-jot} for more on $\Gamma$-isometries), the result on cyclic $\Gamma$-isometries cannot be made to follow from the corresponding result on pairs of isometries. This is an example of the challenges that we alluded to at the end of Section 1.
\end{rem}
%
%\begin{enumerate}
%
%\item The Toeplitz corona theorem on the bidisk (Theorem 11.65 in \cite{ag-Mc}) implies the equivalence of parts (1), $(2')$ and $(3')$ in Theorem 3 while the two new characterizations (parts (2) and (3)) in terms of the admissible kernels on the symmetrized bidisk and an Agler-type decomposition for the symmetrized bidisk cannot be obtained from the bidisk results.
%
%\item

\end{document}